\documentclass[11pt]{amsart}
\usepackage{mathrsfs,amsmath,amssymb,amsthm}
\usepackage{comment}
\newtheorem{theorem}{Theorem}

\newtheorem{Example}{Example}
\newtheorem{proposition}{Proposition}
\newtheorem{lemma}{Lemma}
\newtheorem{corollary}{Corollary}

\title{On the number of cusps of deformations of complex polynomials}
\author{Kazumasa Inaba} 
\address{Mathematical Institute, Tohoku University, Sendai 980-8578, Japan}
\email{kazumasa.inaba.q3@dc.tohoku.ac.jp}

\begin{document}
\renewcommand{\thefootnote}{\fnsymbol{footnote}}
\footnote[0]{2010\textit{ Mathematics Subject Classification}.
Primary 57R45; Secondary: 58K05, 58K60.
}


\maketitle


\begin{abstract}
Let $f$ be a $1$-variable complex polynomial such that $f$ has a singularity at the origin. 
In the present paper, we show that there exists a deformation $f_{t}$ of $f$ which has only 
fold singularities and cusps as singularities of a real polynomial map from $\Bbb{R}^{2}$ to $\Bbb{R}^{2}$. 
We then calculate the number of cusps of $f_t$ in a sufficiently small neighborhood of the origin and 
estimate the number of cusps of $f_t$ in $\Bbb{R}^{2}$. 
\end{abstract}

\section{Introduction}
Let $f : \Bbb{R}^{2} \rightarrow \Bbb{R}^{2}$ be a smooth map 
which has fold singularities and cusps as singularities. 
We call such a map an \textit{excellent map}. 
In \cite{W}, Whitney showed that the set of excellent maps is dense in 
$C^{\infty}(\Bbb{R}^{2}, \Bbb{R}^{2})$. 

It's known that there is a relation between 
the topology of surfaces and 
the topology of the critical locus of a map. 
Quine \cite{Q} and Fukuda--Ishikawa \cite{FI} studied the number of cusps of stable maps 
between oriented $2$-manifolds. 
The degree of cusps of a stable map is determined by the topological degree of a stable map and the Euler 
characteristics of surfaces. 
Fukuda and Ishikawa also studied the number of cusps of stable perturbations of generic map germs \cite{FI}. 
They showed the number of cusps modulo $2$ is a topological invariant of generic map germs. 
Moreover, the number of cusps modulo $2$ depends only on the topology of surfaces. 
Krzy\.{z}anowska and Szafraniec gave 
a criterion to determine if a polynomial map is an excellent map or not \cite{KS}. 
They also gave an algorithm to compute the number of cusps of generic polynomial maps. 
In \cite{S}, Szafraniec considered bifurcations of cusps of families of plane-to-plane maps and 
presented an algebraic method for computing the number of cusps of analytic families. 
In holomorphic case, Gaffney and Mond gave an algebraic formula to count the number of cusps and nodes 
of a generic perturbation of finitely determined holomorphic map germs from $(\Bbb{C}^{2}, \mathbf{o})$ 
to $(\Bbb{C}^{2}, \mathbf{o})$, where $\mathbf{o}$ is the origin of $\Bbb{C}^{2}$ \cite{GM}. 
Farnik, Jelonek and Ruas described the number of cusps and nodes of generic complex polynomial maps \cite{FJR}. 

In this paper, we study certain deformations of complex polynomials. 
We calculate explicitly the number of cusps of deformations by using multiplicities of singularities of 
complex polynomials. 
Let $f(z)$ be a complex polynomial such that $f(0) = 0$. We identify $\Bbb{C}$ with $\Bbb{R}^{2}$. 
Then $f(z)$ defines a real polynomial map 
\[
f :\Bbb{R}^{2} \rightarrow \Bbb{R}^{2}, \ \ \ (x, y) \mapsto (\Re f(x,y), \Im f(x,y)),
\]
where $z = x + \sqrt{-1}y$. 
Assume that the origin $0$ of $\Bbb{C}$ is a singularity of $f$. 
We define a \textit{linear deformation $f_t$ of $f$} as follows: 
\[
f_{t}(z) := f(z) + t(a+ib)\bar{z}, 
\]
where $a, b, t \in \Bbb{R}, i = \sqrt{-1}$ and $0 < \lvert t\rvert \ll 1$. 
Note that a linear deformation $f_t$ of $f$ is not a complex polynomial, 
but is a $1$-variable mixed polynomial in the sense of Oka \cite{O}. 
We now regard a mixed polynomial map $f_{t} : \Bbb{C} \rightarrow \Bbb{C}$ 
as a real polynomial map $(\Re f_{t}, \Im f_{t}) : \Bbb{R}^{2} \rightarrow \Bbb{R}^{2}$. 
If $f(z)=z^{n}$, Fukuda and Ishikawa showed that the number of cusps of a linear deformation of $f$ 
is congruent to $n+1$ modulo $2$, see \cite[Example 2.3]{FI}. 
In \cite{IIKT}, 
the author, Ishikawa, Kawashima and Nguyen showed that there exist linear deformations of $2$-variable Brieskorn polynomials which are excellent maps and 
we estimated the number of cusps of linear deformations. 

As we will show in Lemma $2$, 
$f_t$ is an excellent map for $0 < \lvert t\rvert \ll 1$ 
if $a$ and $b$ lie outside the union of zero sets of 
analytic functions determined by $a, b$ and $f$. 
In particular, such $a$ and $b$ are generic. 
The main theorem is the following. 
\begin{theorem}
Let $f(z)$ be a complex polynomial and $k$ be the multiplicity of the origin. 
Suppose that $k \geq 2$. 
If a linear deformation $f_{t}$ of $f$ is an excellent map for $0 < \lvert t\rvert \ll 1$, 
then the number of cusps of $f_{t}|_{U}$ is equal to $k+1$,
where $U$ is a sufficiently small neighborhood of the origin. 
\end{theorem}
We estimate the number of cusps of $f_t$ in $\Bbb{R}^{2}$. 
\begin{corollary}
Let $f_{t}$ be a liner deformation of a complex polynomial $f$ in Theorem $1$ and $n = \deg f$. 
Assume that $n \geq 2$. 
Then the number of cusps of $f_{t}$ belongs to 
$[n+1, 3n-3]$. 
\end{corollary}

This paper is organized as follows. In Section $2$ 
we give the definition of excellent maps,  
a~criterion to study generic polynomial mappings and 
introduce the notation of the multiplicity of roots of mixed polynomials. 
In Section $3$ we show the existence of linear deformations which are excellent maps. 
In Section $4$ we prove Theorem $1$. 
In Section $5$ we give an example of a deformation of a complex polynomial which has 
$(n+1)$-cusps and also an example which has $(3n-3)$-cusps.

The author would like to thank and Professor Toshizumi Fukui and 
Professor Masaharu Ishikawa for precious comments and fruitful suggestions.

\section{Preliminaries}
\subsection{Excellent maps}
Let $X$ and $Y$ be $2$-dimensional smooth manifolds. 
A smooth map $f : X \rightarrow Y$ is called an \textit{excellent map} 
if for any $p \in X$, there exist local coordinates $(x, y)$ centered at $p$ and 
local coordinates centered at $f(p)$ such that $f$ is locally described in one of the following forms: 

\begin{enumerate}
\item
$(x, y) \mapsto (x, y)$,
\item
$(x, y) \mapsto (x, \pm y^{2})$,
\item
$(x, y) \mapsto (x, y^{3}+xy)$. 
\end{enumerate}
A point in case $(1)$ is a regular point. Points in cases $(2)$ and $(3)$ are called 
a \textit{fold} and a \textit{cusp}, respectively.

We introduce the bundle $J^{r}(X, Y)$ of $r$-jets and its submanifolds 
$S_{k}(X, Y)$ and $S_{1}^{2}(X, Y)$ for $k= 1, 2$. 
Let $j^{r}f(p)$ be the $r$-jet of $f$ at $p$ and set 
\[
J^{r}(X, Y) := \textstyle\bigcup_{(p, q) \in X \times Y} J^{r}(X, Y, p, q), 
\]
where $J^{r}(X, Y, p, q) = \{ j^{r}f(p) \mid f(p) = q \}$. 
The set $J^{r}(X, Y)$ is called \textit{the bundle of $r$-jets of maps from $X$ into $Y$}. 
The \textit{$r$-extension} $j^{r}f : X \rightarrow J^{r}(X, Y)$ of $f$ 
is defined by 
$p \mapsto j^{r}f(p)$, where $p \in X$.
It is known that $J^{r}(X, Y)$ is a~smooth manifold and 
the $1$-extension $j^{r}f$ of $f$ is a smooth map. 
We define a submanifold
of $J^{1}(X, Y)$ for $k = 1, 2$ as follows: 
\[
S_{k}(X, Y) = \{ j^{1}f(p) \in J^{1}(X, Y) \mid 
\text{rank}\>df_{p} = 2-k \}. 
\]

A smooth map $f : X \rightarrow Y$ is an excellent map 
if and only if 
\begin{enumerate}
\item
$j^{1}f$ is transversal to $S_{1}(X, Y)$ and $S_{2}(X, Y)$, 
\item
$j^{2}f$ is transversal to $S_{1}^{2}(X, Y)$, 
\end{enumerate}
where 
$S_{1}^{2}(X, Y)$ 
is defined as follows: 
\[
S_{1}^{2}(X, Y) = 
\left\{j^{2}f(p) \in J^{2}(X, Y) \;\left|\; 
\begin{aligned}
& j^{1}f(p) \in S_{1}(X, Y), \\  
& j^{1}f~\text{is transversal to}~S_{1}(X, Y)~\text{at}~p, \\
& \text{rank}\>d(f~|~S_{1}(f))(p)~=~0 \\  
\end{aligned}
\right.\right\}. 
\]
It is known that the subset of smooth maps from $X$ to $Y$ which 
are excellent maps 
is open and dense in $C^{\infty}(X, Y)$ topologized with the $C^{\infty}$-topology \cite{L1, W}.

\subsection{Singularities of polynomial maps}
Let $g = (g_{1}, g_{2}): U \rightarrow \Bbb{R}^{2}$ be a polynomial map, where $U$ is an open set. 
Set 
$J = \frac{\partial (g_{1}, g_{2})}{\partial (x, y)}, G_{i} = \frac{\partial (g_{i}, J)}{\partial (x, y)}$ 
for $i=1, 2$. 
We define the algebraic set $G'$ as follows: 
\[
G' := \Bigl\{ (x,y) \in U \mid J(x, y) = G_{1}(x, y) = G_{2}(x, y) = 
\frac{\partial (G_{1}, J)}{\partial (x, y)} = \frac{\partial (G_{2}, J)}{\partial (x, y)} =0 \Bigr\}. 
\]
In \cite[Proposition 2]{KS} and \cite[Proposition 2.2]{S}, 
Krzy\.{z}anowska and Szafraniec showed the following proposition: 
\begin{proposition}
The algebraic set $G'$ 
is empty if and only if 
the set of singularities of $g$ consists of either fold singularities or cusps. 
Moreover, the number of cusps of $g$ is equal to the number of 
$\{(x, y) \in U \mid J(x, y) = G_{1}(x, y) = G_{2}(x, y) = 0\}$.
\end{proposition}

\subsection{Multiplicity with sign}
Set $z = x+iy$. Then a pair of real polynomials $(g_{1}, g_{2})$ defines a mixed polynomial $g(z, \bar{z})$ as follows: 
\begin{equation*}
\begin{split}
g(z, \bar{z}) &= g_{1}(x, y) + ig_{2}(x, y) \\
              &= g_{1}\Bigl(\frac{z + \bar{z}}{2},  \frac{z - \bar{z}}{2i}\Bigr) + ig_{2}\Bigl(\frac{z + \bar{z}}{2},  \frac{z - \bar{z}}{2i}\Bigr). 
\end{split}
\end{equation*}
Then $\frac{\partial g}{\partial z}$ and $\frac{\partial g}{\partial \bar{z}}$ satisfy the following equations: 
\begin{equation*}
\begin{split}
\frac{\partial g}{\partial z} &= 
\frac{1}{2}\Bigl(\frac{\partial g}{\partial x} - i\frac{\partial g}{\partial y}\Bigr) 
= \frac{1}{2}\Bigl(\frac{\partial g_{1}}{\partial x} + \frac{\partial g_{2}}{\partial y}\Bigr) 
+ \frac{i}{2}\Bigl(\frac{\partial g_{2}}{\partial x} - \frac{\partial g_{1}}{\partial y}\Bigr), \\ 
\frac{\partial g}{\partial \bar{z}} &= \frac{1}{2}\Bigl(\frac{\partial g}{\partial x} + i\frac{\partial g}{\partial y}\Bigr) 
= \frac{1}{2}\Bigl(\frac{\partial g_{1}}{\partial x} - \frac{\partial g_{2}}{\partial y}\Bigr) 
+ \frac{i}{2}\Bigl(\frac{\partial g_{2}}{\partial x} + \frac{\partial g_{1}}{\partial y}\Bigr).
\end{split}
\end{equation*}
Suppose that $w$ is a mixed singularity of a mixed polynomial $g$, i.e., 
the gradient vectors of $g_{1}$ and $g_{2}$ at $w$ 
are linearly dependent. Then we have 
\[
\Bigl\lvert \frac{\partial g}{\partial z}(w)\Bigr\rvert = \Bigl\lvert \frac{\partial g}{\partial \bar{z}}(w)\Bigr\rvert, 
\]
see \cite{O}. Let $\alpha \in \Bbb{C}$ be an isolated root of $g(z, \bar{z})=0$. 
Put 
\[
S^{1}_{\varepsilon}(\alpha) := \{ z \in \Bbb{C} \mid \lvert z - \alpha \rvert = \varepsilon \}, 
\]
where $\varepsilon$ is a sufficiently small positive real number. 
We define the \textit{multiplicity with the sign of the root $\alpha$} by the mapping degree of the normalized function 
\[
\frac{g}{\lvert g\rvert} : S^{1}_{\varepsilon}(\alpha) \rightarrow S^{1}. 
\]
We denote the multiplicity with the sign of the root $\alpha$ by $m_{s}(g, \alpha)$. 

We say that $\alpha$ is a \textit{positive simple root} if $\alpha$ satisfies 
\[
\Bigl\lvert \frac{\partial g}{\partial z}(\alpha)\Bigr\rvert > 
\Bigl\lvert \frac{\partial g}{\partial \bar{z}}(\alpha)\Bigr\rvert. 
\]
Similarly, $\alpha$ is a \textit{negative simple root} if $\alpha$ satisfies 
\[
\Bigl\lvert \frac{\partial g}{\partial z}(\alpha)\Bigr\rvert < 
\Bigl\lvert \frac{\partial g}{\partial \bar{z}}(\alpha)\Bigr\rvert. 
\]
In \cite[Proposition 15]{O}, $\alpha$ is a positive (resp. negative) simple root 
if and only if $m_{s}(g, \alpha) = 1$ (resp. $m_{s}(g, \alpha) = -1$). 

Consider a bifurcation family $g_{t}(z, \bar{z})=0$ for $g_{0} = g$ and $t \in \Bbb{R}$. 
Let $\{P_{1}(t), \dots, P_{\nu}(t)\}$ 
be the roots of $g_{t}(z, \bar{z})=0$ which are bifurcating from $z=\alpha$. 
Then we have 
\[
\textstyle\sum_{j=1}^{\nu}m_{s}(g_{t}, P_{j}(t)) = m_{s}(g, \alpha), 
\]
see \cite[Proposition 16]{O}.

\section{The existence of linear deformations which are excellent maps}
Let $f(z)$ be a complex polynomial.
Assume that $f(0)=0$ and the origin of $\Bbb{C}$ is a singularity of $f$. 
Set $f_{1} = \Re f$ and $f_{2} = \Im f$. 
We take $a, b \in \Bbb{R}$. Then 
a linear deformation $f_t$ of $f$ is defined by 
$f_{t}(z) = f(z) + t(a+ib)\bar{z}$, where 
$0 < \lvert t \rvert \ll 1$. 
Note that $f_t$ is equal to 
\begin{equation*}
\begin{split}
f_{t}(z) &= f(z) + t(a+ib)\bar{z} \\
         &= f_{1}(z) + t(ax+by) + i\{f_{2}(z) + t(bx-ay)\}. 
\end{split}
\end{equation*}
Then $f_t$ defines a real polynomial map from $\Bbb{R}^{2}$ to $\Bbb{R}^{2}$ as follows: 
\[
f_{t} : \Bbb{R}^{2} \rightarrow \Bbb{R}^{2}, \ \ \ (x, y) \mapsto (f_{1}(x,y) + t(ax+by), f_{2}(x,y) + t(bx-ay)).
\]
We calculate $J, G_{1}$ and $G_{2}$ of $f_{t}$. 
By the Cauchy--Riemann equations 
$\frac{\partial f_{2}}{\partial x} = -\frac{\partial f_{1}}{\partial y}$ and 
$\frac{\partial f_{2}}{\partial y} = \frac{\partial f_{1}}{\partial x}$, 
$J$ is modified as 
\begin{equation*}
\begin{split}
J &= \det \begin{pmatrix}
         \frac{\partial f_{1}}{\partial x} + ta & \frac{\partial f_{1}}{\partial y} + tb \\
         \frac{\partial f_{2}}{\partial x} + tb & \frac{\partial f_{2}}{\partial y} - ta
         \end{pmatrix} \\
  &= \det \begin{pmatrix}
         \frac{\partial f_{1}}{\partial x} + ta & \frac{\partial f_{1}}{\partial y} + tb \\
         -\frac{\partial f_{1}}{\partial y} + tb & \frac{\partial f_{1}}{\partial x} - ta
         \end{pmatrix} \\
  &= \Bigl(\frac{\partial f_{1}}{\partial x}\Bigr)^{2} + 
  \Bigl(\frac{\partial f_{1}}{\partial y}\Bigr)^{2} - t^{2}(a^{2} + b^{2})\\
  &= \Bigl\lvert \frac{\partial f}{\partial z}\Bigr\rvert^{2} - t^{2}(a^{2} + b^{2}).
\end{split}
\end{equation*}
Since $f$ is a harmonic function, 
$\frac{\partial f_{1}}{\partial x\partial x} = -\frac{\partial f_{1}}{\partial y\partial y}$. 
Then we have 
\begin{equation*}
\begin{split}
G_{1} &= \det \begin{pmatrix}
         \frac{\partial f_{1}}{\partial x} + ta & \frac{\partial f_{1}}{\partial y} + tb \\
         \frac{\partial J}{\partial x}  & \frac{\partial J}{\partial y}
         \end{pmatrix} \\
  &= 2\Bigl(\frac{\partial f_{1}}{\partial x} + ta\Bigr)
  \Bigl(\frac{\partial f_{1}}{\partial x}\frac{\partial^{2} f_{1}}{\partial x \partial y} 
+ \frac{\partial f_{1}}{\partial y}\frac{\partial^{2} f_{1}}{\partial y \partial y}\Bigr)  
  -2\Bigl(\frac{\partial f_{1}}{\partial y} + tb\Bigr)
  \Bigl(-\frac{\partial f_{1}}{\partial x}\frac{\partial^{2} f_{1}}{\partial y \partial y} 
+ \frac{\partial f_{1}}{\partial y}\frac{\partial^{2} f_{1}}{\partial x \partial y}\Bigr) \\
&= 2\biggl(\Bigl(\frac{\partial f_{1}}{\partial x}\Bigr)^{2} - \Bigl(\frac{\partial f_{1}}{\partial y}\Bigr)^{2}\biggr)\frac{\partial^{2} f_{1}}{\partial x \partial y} 
+ 4\frac{\partial f_{1}}{\partial x}\frac{\partial f_{1}}{\partial y}\frac{\partial^{2} f_{1}}{\partial y \partial y}  \\
&+ 2t\biggl\{a\Bigl(\frac{\partial f_{1}}{\partial x}\frac{\partial^{2} f_{1}}{\partial x \partial y} + 
\frac{\partial f_{1}}{\partial y}\frac{\partial^{2} f_{1}}{\partial y \partial y}\Bigr) 
-b\Bigl(-\frac{\partial f_{1}}{\partial x}\frac{\partial^{2} f_{1}}{\partial y \partial y} + 
\frac{\partial f_{1}}{\partial y}\frac{\partial^{2} f_{1}}{\partial x \partial y}\Bigr)\biggr\}.
\end{split}
\end{equation*}
By the same argument, $G_2$ is equal to 
\begin{equation*}
\begin{split}
G_{2} &= \det \begin{pmatrix}
         -\frac{\partial f_{1}}{\partial y} + tb & \frac{\partial f_{1}}{\partial x} - ta \\
         \frac{\partial J}{\partial x}  & \frac{\partial J}{\partial y}
         \end{pmatrix}\\
   &=2\Bigl(-\frac{\partial f_{1}}{\partial y} + tb\Bigr)
  \Bigl(\frac{\partial f_{1}}{\partial x}\frac{\partial^{2} f_{1}}{\partial x \partial y} 
+ \frac{\partial f_{1}}{\partial y}\frac{\partial^{2} f_{1}}{\partial y \partial y}\Bigr)  
  -2\Bigl(\frac{\partial f_{1}}{\partial x} - ta\Bigr)
  \Bigl(-\frac{\partial f_{1}}{\partial x}\frac{\partial^{2} f_{1}}{\partial y \partial y} 
+ \frac{\partial f_{1}}{\partial y}\frac{\partial^{2} f_{1}}{\partial x \partial y}\Bigr) \\
&= 2\biggl(\Bigl(\frac{\partial f_{1}}{\partial x}\Bigr)^{2} - \Bigl(\frac{\partial f_{1}}{\partial y}\Bigr)^{2}\biggr)\frac{\partial^{2} f_{1}}{\partial y \partial y} 
- 4\frac{\partial f_{1}}{\partial x}\frac{\partial f_{1}}{\partial y}\frac{\partial^{2} f_{1}}{\partial x \partial y}  \\
&+ 2t\biggl\{a\Bigl(-\frac{\partial f_{1}}{\partial x}\frac{\partial^{2} f_{1}}{\partial y \partial y} + 
\frac{\partial f_{1}}{\partial y}\frac{\partial^{2} f_{1}}{\partial x \partial y}\Bigr) 
+b\Bigl(\frac{\partial f_{1}}{\partial x}\frac{\partial^{2} f_{1}}{\partial x \partial y} + 
\frac{\partial f_{1}}{\partial y}\frac{\partial^{2} f_{1}}{\partial y \partial y}\Bigr)\biggr\}. 
\end{split}
\end{equation*}
If $G_{1}$ and $G_{2}$ are equal to $0$ at $(x, y)$, then $(x, y)$ satisfies the following equation: 
\begin{equation*}
\begin{split}
&\biggl\{\biggl(\Bigl(\frac{\partial f_{1}}{\partial x}\Bigr)^{2} - \Bigl(\frac{\partial f_{1}}{\partial y}\Bigr)^{2}\biggr)\frac{\partial^{2} f_{1}}{\partial x \partial y} 
+ 2\frac{\partial f_{1}}{\partial x}\frac{\partial f_{1}}{\partial y}\frac{\partial^{2} f_{1}}{\partial y \partial y}\biggr\} \\
\times &\biggl\{a\Bigl(-\frac{\partial f_{1}}{\partial x}\frac{\partial^{2} f_{1}}{\partial y \partial y} + 
\frac{\partial f_{1}}{\partial y}\frac{\partial^{2} f_{1}}{\partial x \partial y}\Bigr) + 
b\Bigl(\frac{\partial f_{1}}{\partial x}\frac{\partial^{2} f_{1}}{\partial x \partial y} + 
\frac{\partial f_{1}}{\partial y}\frac{\partial^{2} f_{1}}{\partial y \partial y}\Bigr)\biggr\} \\ 
= &\biggl\{\biggl(\Bigl(\frac{\partial f_{1}}{\partial x}\Bigr)^{2} - \Bigl(\frac{\partial f_{1}}{\partial y}\Bigr)^{2}\biggr)\frac{\partial^{2} f_{1}}{\partial y \partial y} 
- 2\frac{\partial f_{1}}{\partial x}\frac{\partial f_{1}}{\partial y}\frac{\partial^{2} f_{1}}{\partial x \partial y}\biggr\} \\
\times &\biggl\{a\Bigl(\frac{\partial f_{1}}{\partial x}\frac{\partial^{2} f_{1}}{\partial x \partial y} + 
\frac{\partial f_{1}}{\partial y}\frac{\partial^{2} f_{1}}{\partial y \partial y}\Bigr) - 
b\Bigl(-\frac{\partial f_{1}}{\partial x}\frac{\partial^{2} f_{1}}{\partial y \partial y} + 
\frac{\partial f_{1}}{\partial y}\frac{\partial^{2} f_{1}}{\partial x \partial y}\Bigr)\biggr\}.
\end{split}
\end{equation*}
Hence we have 
\begin{equation}
\begin{split}
a\biggl[&\biggl(-3\Bigl(\frac{\partial f_{1}}{\partial x}\Bigr)^{2} + \Bigl(\frac{\partial f_{1}}{\partial y}\Bigr)^{2}\biggr)
\frac{\partial f_{1}}{\partial y}\biggl\{\Bigl(\frac{\partial^{2} f_{1}}{\partial y\partial y}\Bigr)^{2} - \Bigl(\frac{\partial^{2} f_{1}}{\partial x\partial y}\Bigr)^{2}\biggr\} \\
&+2\biggl(-\Bigl(\frac{\partial f_{1}}{\partial x}\Bigr)^{2} + 3\Bigl(\frac{\partial f_{1}}{\partial y}\Bigr)^{2}\biggr)
\frac{\partial f_{1}}{\partial x}\frac{\partial^{2} f_{1}}{\partial x\partial y}
\frac{\partial^{2} f_{1}}{\partial y\partial y}\biggr] \\
+ 
b\biggl[&\biggl(-\Bigl(\frac{\partial f_{1}}{\partial x}\Bigr)^{2} + 3\Bigl(\frac{\partial f_{1}}{\partial y}\Bigr)^{2}\biggr)
\frac{\partial f_{1}}{\partial x}
\biggl\{\Bigl(\frac{\partial^{2} f_{1}}{\partial y\partial y}\Bigr)^{2} - \Bigl(\frac{\partial^{2} f_{1}}{\partial x\partial y}\Bigr)^{2}\biggr\}\\ 
&-2\biggl(-3\Bigl(\frac{\partial f_{1}}{\partial x}\Bigr)^{2} + \Bigl(\frac{\partial f_{1}}{\partial y}\Bigr)^{2}\biggr)
\frac{\partial f_{1}}{\partial y}
\frac{\partial^{2} f_{1}}{\partial x\partial y}\frac{\partial^{2} f_{1}}{\partial y\partial y}\biggr] \\
&= 0. 
\end{split}
\end{equation}
Set real polynomials $\phi_{1}, \phi_{2}$ and $\Phi$ as follows: 
\begin{equation*}
\begin{split}
\phi_{1} := &\biggl(-3\Bigl(\frac{\partial f_{1}}{\partial x}\Bigr)^{2} + \Bigl(\frac{\partial f_{1}}{\partial y}\Bigr)^{2}\biggr)
\frac{\partial f_{1}}{\partial y}\biggl\{\Bigl(\frac{\partial^{2} f_{1}}{\partial y\partial y}\Bigr)^{2} - \Bigl(\frac{\partial^{2} f_{1}}{\partial x\partial y}\Bigr)^{2}\biggr\} \\
&+2\biggl(-\Bigl(\frac{\partial f_{1}}{\partial x}\Bigr)^{2} + 3\Bigl(\frac{\partial f_{1}}{\partial y}\Bigr)^{2}\biggr)
\frac{\partial f_{1}}{\partial x}\frac{\partial^{2} f_{1}}{\partial x\partial y}\frac{\partial^{2} f_{1}}{\partial y\partial y},   \\
\phi_{2} := &\biggl(-\Bigl(\frac{\partial f_{1}}{\partial x}\Bigr)^{2} + 3\Bigl(\frac{\partial f_{1}}{\partial y}\Bigr)^{2}\biggr)
\frac{\partial f_{1}}{\partial x}
\biggl\{\Bigl(\frac{\partial^{2} f_{1}}{\partial y\partial y}\Bigr)^{2} - \Bigl(\frac{\partial^{2} f_{1}}{\partial x\partial y}\Bigr)^{2}\biggr\} \\
&-2\biggl(-3\Bigl(\frac{\partial f_{1}}{\partial x}\Bigr)^{2} + \Bigl(\frac{\partial f_{1}}{\partial y}\Bigr)^{2}\biggr)
\frac{\partial f_{1}}{\partial y}
\frac{\partial^{2} f_{1}}{\partial x\partial y}\frac{\partial^{2} f_{1}}{\partial y\partial y}, \\ 
\Phi &:= a\phi_{1} + b\phi_{2}. 
\end{split}
\end{equation*}
Suppose that $G_{1}$ and $G_{2}$ are equal to $0$ at $(x, y)$. 
By equation $(1)$ and the definitions of $\phi_{1}, \phi_{2}$ and $\Phi$, $\Phi(x, y)$ is also equal to $0$. 
To show the existence of linear deformations which are excellent maps, 
we consider the intersection of $\phi_{1}^{-1}(0)$ and $\phi_{2}^{-1}(0)$. 
\begin{lemma}
Let $U$ be a sufficiently small neighborhood of the origin $0$ of $\Bbb{C}$.
Assume that $U$ satisfies 
$\{w \in U \mid \frac{\partial f}{\partial z}(w)=~0\} = \{0\}$ and 
$\{w \in U \mid \frac{\partial^{2} f}{\partial z\partial z}(w)=0\} \subset \{0\}$. 
Then the intersection of $\phi_{1}^{-1}(0), \phi_{2}^{-1}(0)$ and $U$ is equal to~$\{0\}$. 
\end{lemma}
\begin{proof}
Let $(x, y)$ be a point of $\phi_{1}^{-1}(0) \cap \phi_{2}^{-1}(0)$. 
Assume that $(x, y)$ satisfies $\frac{\partial f_{1}}{\partial x}(x,y) = 0$ or 
$\Bigl(\frac{\partial f_{1}}{\partial x}(x,y)\Bigr)^{2} - 3\Bigl(\frac{\partial f_{1}}{\partial y}(x,y)\Bigr)^{2} = 0$. 
Then we have 
\begin{equation*}
\begin{split}
\Bigl(\frac{\partial f_{1}}{\partial y}(x,y)\Bigr)^{3}\biggl\{\Bigl(\frac{\partial^{2} f_{1}}{\partial y\partial y}(x,y)\Bigr)^{2} - \Bigl(\frac{\partial^{2} f_{1}}{\partial x\partial y}(x,y)\Bigr)^{2}\biggr\} &= 0, \\
\Bigl(\frac{\partial f_{1}}{\partial y}(x,y)\Bigr)^{3}\frac{\partial^{2} f_{1}}{\partial x\partial y}(x,y)\frac{\partial^{2} f_{1}}{\partial y\partial y}(x,y) &=0.
\end{split}
\end{equation*}
By the above equations, $(x, y)$ satisfies $\frac{\partial f_{1}}{\partial y}(x,y) = 0$ or 
$\frac{\partial^{2} f_{1}}{\partial x\partial y}(x,y) = \frac{\partial^{2} f_{1}}{\partial y\partial y}(x,y) = 0$. 
So $(x,y)$ belongs to 
$\{ (x, y) \in U \mid \frac{\partial f_{1}}{\partial x}(x,y) = \frac{\partial f_{1}}{\partial y}(x,y) = 0 \}$ or 
$\{ (x, y) \in U \mid \frac{\partial^{2} f_{1}}{\partial x\partial y}(x,y) = \frac{\partial^{2} f_{1}}{\partial y\partial y}(x,y) = 0\}$. 
By the assumption of $U$, $(x, y)$ is equal to~$0$. 
Suppose that $(x, y)$ satisfies $\frac{\partial f_{1}}{\partial y}(x,y) = 0$ or 
$-3\Bigl(\frac{\partial f_{1}}{\partial x}(x,y)\Bigr)^{2} + \Bigl(\frac{\partial f_{1}}{\partial y}(x,y)\Bigr)^{2} = 0$. 
By the same argument, we can check that $(x,y)$ is equal to~$0$. 

We assume that $(x, y)$ satisfies 
$\frac{\partial f_{1}}{\partial x}\frac{\partial f_{1}}{\partial y} \neq 0$ and 
$\biggl\{\Bigl(\frac{\partial f_{1}}{\partial x}\Bigr)^{2} - 3\Bigl(\frac{\partial f_{1}}{\partial y}\Bigr)^{2}\biggr\} 
\biggl\{3\Bigl(\frac{\partial f_{1}}{\partial x}\Bigr)^{2} - \Bigl(\frac{\partial f_{1}}{\partial y}\Bigr)^{2}\biggr\} \neq 0$ 
on $U$. 
Then $(x, y) \in \phi_{1}^{-1}(0) \cap \phi_{2}^{-1}(0)$ satisfies 
\begin{equation}
\begin{split}
2\frac{\partial^{2} f_{1}}{\partial x\partial y}\frac{\partial^{2} f_{1}}{\partial y\partial y} &= 
\frac{\biggl(-3\Bigl(\frac{\partial f_{1}}{\partial x}\Bigr)^{2} + 
\Bigl(\frac{\partial f_{1}}{\partial x}\Bigr)^{2}\biggr)\frac{\partial f_{1}}{\partial y}\biggl\{\Bigl(\frac{\partial^{2} f_{1}}{\partial y\partial y}\Bigr)^{2} - \Bigl(\frac{\partial^{2} f_{1}}{\partial x\partial y}\Bigr)^{2}\biggr\}}
{-\biggl(-\Bigl(\frac{\partial f_{1}}{\partial x}\Bigr)^{2} + 3\Bigl(\frac{\partial f_{1}}{\partial y}\Bigr)^{2}\biggr)\frac{\partial f_{1}}{\partial x}}, \\
2\frac{\partial^{2} f_{1}}{\partial x\partial y}\frac{\partial^{2} f_{1}}{\partial y\partial y} &= 
\frac{\biggl(-\Bigl(\frac{\partial f_{1}}{\partial x}\Bigr)^{2} + 
3\Bigl(\frac{\partial f_{1}}{\partial y}\Bigr)^{2}\biggr)\frac{\partial f_{1}}{\partial x}\biggl\{\Bigl(\frac{\partial^{2} f_{1}}{\partial y\partial y}\Bigr)^{2} - \Bigl(\frac{\partial^{2} f_{1}}{\partial x\partial y}\Bigr)^{2}\biggr\}}
{\biggl(-3\Bigl(\frac{\partial f_{1}}{\partial x}\Bigr)^{2} + \Bigl(\frac{\partial f_{1}}{\partial y}\Bigr)^{2}\biggr)\frac{\partial f_{1}}{\partial y}}. 
\end{split}
\end{equation}
By the above equations, $(x, y)$ satisfies the following equation: 
\begin{equation*}
\begin{split}
&\Biggl\{\biggl(-3\Bigl(\frac{\partial f_{1}}{\partial x}\Bigr)^{2} + 
\Bigl(\frac{\partial f_{1}}{\partial y}\Bigr)^{2}\biggr)^{2}\Bigl(\frac{\partial f_{1}}{\partial y}\Bigr)^{2}
+\biggl(-\Bigl(\frac{\partial f_{1}}{\partial x}\Bigr)^{2} + 3\Bigl(\frac{\partial f_{1}}{\partial y}\Bigr)^{2}\biggr)^{2}\Bigl(\frac{\partial f_{1}}{\partial x}\Bigr)^{2}\Biggr\} \\ 
&\times
\biggl\{\Bigl(\frac{\partial^{2} f_{1}}{\partial y\partial y}\Bigr)^{2} - \Bigl(\frac{\partial^{2} f_{1}}{\partial x\partial y}\Bigr)^{2}\biggr\} 
= 0. 
\end{split}
\end{equation*}
Thus 
$\Bigl(\frac{\partial^{2} f_{1}}{\partial y\partial y}(x,y)\Bigr)^{2} - \Bigl(\frac{\partial^{2} f_{1}}{\partial x\partial y}(x,y)\Bigr)^{2}$ 
is equal to $0$. 
By equation $(2)$, the second differentials 
$\frac{\partial^{2} f_{1}}{\partial y\partial y}(x,y)$ and $\frac{\partial^{2} f_{1}}{\partial x\partial y}(x,y)$ 
of $f_{1}$ are equal to $0$. 
By the assumption of $U$, 
the intersection $\phi_{1}^{-1}(0) \cap \phi_{2}^{-1}(0) \cap U$ is equal to $\{0\}$. 
\end{proof}

To study singularities of $f_t$, we define the mixed polynomial $G_{t}$ as follows: 
\begin{equation*}
\begin{split}
G_{t} &:= G_{1} + iG_{2} \\
  &=  \det \begin{pmatrix}
         \frac{\partial f_{1}}{\partial x} + ta & \frac{\partial f_{1}}{\partial y} + tb \\   
         \frac{\partial J}{\partial x}  & \frac{\partial J}{\partial y}
         \end{pmatrix}
         + i\det \begin{pmatrix}
         -\frac{\partial f_{1}}{\partial y} + tb & \frac{\partial f_{1}}{\partial x} - ta \\
         \frac{\partial J}{\partial x}  & \frac{\partial J}{\partial y} 
         \end{pmatrix}\\
  &= \biggl(\frac{\partial f_{1}}{\partial x} + ta + i\Bigl(-\frac{\partial f_{1}}{\partial y} + tb\Bigr)\biggr)\frac{\partial J}{\partial y} - 
     \biggl(\frac{\partial f_{1}}{\partial y} + tb + i\Bigl(\frac{\partial f_{1}}{\partial x} - ta\Bigr)\biggr)\frac{\partial J}{\partial x} \\
    &= \Bigl(\frac{\partial f}{\partial z} + t(a+ib)\Bigr)\frac{\partial J}{\partial y} - 
  i\Bigl(\frac{\partial f}{\partial z} - t(a+ib)\Bigr)\frac{\partial J}{\partial x}. 
\end{split}
\end{equation*}
Since $\frac{\partial J}{\partial z}$ is equal to 
$\frac{1}{2}\Bigl(\frac{\partial J}{\partial x} - i\frac{\partial J}{\partial y}\Bigr)$, 
$\frac{\partial J}{\partial x}$ and $\frac{\partial J}{\partial y}$ are equal to 
\begin{equation*}
\begin{split}
\frac{\partial J}{\partial x} &= 2\Re \frac{\partial J}{\partial z} 
= 2\Re \frac{\partial^{2} f}{\partial z \partial z}\overline{\frac{\partial f}{\partial z}} = 
\frac{\partial^{2} f}{\partial z \partial z}\overline{\frac{\partial f}{\partial z}} 
+ \overline{\frac{\partial^{2} f}{\partial z \partial z}}\frac{\partial f}{\partial z},  \\
\frac{\partial J}{\partial y} &= -2\Im \frac{\partial J}{\partial z} = 
-2\Im \frac{\partial^{2} f}{\partial z \partial z}\overline{\frac{\partial f}{\partial z}} 
= i\Bigl(\frac{\partial^{2} f}{\partial z \partial z}\overline{\frac{\partial f}{\partial z}} 
- \overline{\frac{\partial^{2} f}{\partial z \partial z}}\frac{\partial f}{\partial z}\Bigr), 
\end{split}
\end{equation*}
where $z = x+iy$. 
Thus $G_{t}$ is equal to 
\begin{equation*}
\begin{split}
  &i\Bigl(\frac{\partial f}{\partial z} + t(a+ib)\Bigr)\Bigl(\frac{\partial^{2} f}{\partial z \partial z}\overline{\frac{\partial f}{\partial z}} 
- \overline{\frac{\partial^{2} f}{\partial z \partial z}}\frac{\partial f}{\partial z}\Bigr) - 
i\Bigl(\frac{\partial f}{\partial z} - t(a+ib)\Bigr)\Bigl( \frac{\partial^{2} f}{\partial z \partial z}\overline{\frac{\partial f}{\partial z}} 
+ \overline{\frac{\partial^{2} f}{\partial z \partial z}}\frac{\partial f}{\partial z}\Bigr) \\
 &= -2i\Bigl(\frac{\partial f}{\partial z}\Bigr)^{2}\overline{\frac{\partial^{2} f}{\partial z \partial z}} 
 + 2ti(a+ib)\frac{\partial^{2} f}{\partial z \partial z}\overline{\frac{\partial f}{\partial z}}.
\end{split}
\end{equation*}
Suppose that $z$ satisfies $G_{t}(z)=0$ and 
$\frac{\partial f}{\partial z}(z)\frac{\partial^{2} f}{\partial z \partial z}(z) \neq 0$. 
By the above equation, $z$ satisfies $J(z) = 0$. 
Since the multiplicity $k$ of $f$ at the origin is greater than $1$, 
$G_{t}(0) = 0$ and 
$\frac{\partial f}{\partial z}(0)\frac{\partial^{2} f}{\partial z \partial z}(0) = 0$. 
Thus we have 
\begin{equation*}
\begin{split} 
&\biggl\{z \in U \mid G_{t}(z)=0, \frac{\partial f}{\partial z}(z) \neq 0,  
\frac{\partial^{2} f}{\partial z \partial z}(z) \neq 0\biggr\} \\
= &\{z \in U \setminus \{0\} \mid G_{t}(z)=0\} 
\subset J^{-1}(0). 
\end{split}
\end{equation*}
Similarly, we define the following mixed polynomial: 
\begin{equation*}
\begin{split}
H_{t} &:= \det \begin{pmatrix}
         \frac{\partial G_{1}}{\partial x} & \frac{\partial G_{1}}{\partial y}  \\   
         \frac{\partial J}{\partial x}  & \frac{\partial J}{\partial y}
         \end{pmatrix}
         + i\det \begin{pmatrix}
         \frac{\partial G_{2}}{\partial x} & \frac{\partial G_{2}}{\partial y}  \\
         \frac{\partial J}{\partial x}  & \frac{\partial J}{\partial y} 
         \end{pmatrix}\\
  &= \Bigl(\frac{\partial G_{1}}{\partial x} +i\frac{\partial G_{2}}{\partial x}\Bigr)\frac{\partial J}{\partial y}
  - \Bigl(\frac{\partial G_{1}}{\partial y} +i\frac{\partial G_{2}}{\partial y}\Bigr)\frac{\partial J}{\partial x}. \\
\end{split}
\end{equation*}
The differentials of $G_{t}$ satisfy the following equations: 
\[
\frac{\partial G_{t}}{\partial z} = \frac{1}{2}\Bigl(\frac{\partial G_{1}}{\partial x} + 
\frac{\partial G_{2}}{\partial y}\Bigr) 
+ \frac{i}{2}\Bigl(\frac{\partial G_{2}}{\partial x} - 
\frac{\partial G_{1}}{\partial y}\Bigr),  \ \ \ 
\frac{\partial G_{t}}{\partial \bar{z}} = \frac{1}{2}\Bigl(\frac{\partial G_{1}}{\partial x} - 
\frac{\partial G_{2}}{\partial y}\Bigr) 
+ \frac{i}{2}\Bigl(\frac{\partial G_{2}}{\partial x} + 
\frac{\partial G_{1}}{\partial y}\Bigr).
\]
Then we have 
\begin{equation*}
\begin{split} 
  H_{t}&= \Bigl(\frac{\partial G_{t}}{\partial z} + \frac{\partial G_{t}}{\partial \bar{z}}\Bigr)\frac{\partial J}{\partial y}
  -i\Bigl(\frac{\partial G_{t}}{\partial z} - \frac{\partial G_{t}}{\partial \bar{z}}\Bigr)\frac{\partial J}{\partial x} \\
  &= \frac{\partial G_{t}}{\partial z}\Bigl(\frac{\partial J}{\partial y} -i\frac{\partial J}{\partial x}\Bigr) + 
  \frac{\partial G_{t}}{\partial \bar{z}}\Bigl(\frac{\partial J}{\partial y} + i\frac{\partial J}{\partial x}\Bigr). \\
\end{split}
\end{equation*}
Since $\frac{\partial J}{\partial y} -i\frac{\partial J}{\partial x} = -2i\frac{\partial J}{\partial \bar{z}}$ and 
$\frac{\partial J}{\partial y} + i\frac{\partial J}{\partial x} = 2i\frac{\partial J}{\partial z}$, 
$H_{t}$ is equal to 
\begin{equation*}
\begin{split} 
  H_{t}
  &= -2i\frac{\partial G_{t}}{\partial z}\frac{\partial J}{\partial \bar{z}} 
  + 2i\frac{\partial G_{t}}{\partial \bar{z}}\frac{\partial J}{\partial z} \\
  &= -2i\biggl\{-4i\frac{\partial f}{\partial z}\Bigl\lvert \frac{\partial^{2} f}{\partial z \partial z}\Bigr\rvert^{2} 
  + 2ti(a+ib)\frac{\partial^{3} f}{\partial z \partial z \partial z}\overline{\frac{\partial f}{\partial z}}\biggr\}
  \frac{\partial f}{\partial z}\overline{\frac{\partial^{2} f}{\partial z\partial z}} \\
  &+2i\biggl\{-2i\Bigl(\frac{\partial f}{\partial z}\Bigr)^{2}\overline{\frac{\partial^{3} f}{\partial z \partial z \partial z}} 
  + 2ti(a+ib)\Bigl\lvert \frac{\partial^{2} f}{\partial z \partial z}\Bigr\rvert^{2}\biggr\}\frac{\partial^{2} f}{\partial z \partial z}\overline{\frac{\partial f}{\partial z}} \\
  &= -4\Bigl(\frac{\partial f}{\partial z}\Bigr)^{2}\frac{\partial^{2} f}{\partial z \partial z}
  \biggl\{\overline{2\Bigl(\frac{\partial^{2} f}{\partial z\partial z}\Bigr)^{2} - \frac{\partial f}{\partial z}\frac{\partial^{3} f}{\partial z \partial z \partial z}}\biggr\} \\
  &+ 4t(a+ib)\overline{\frac{\partial f}{\partial z}\frac{\partial^{2} f}{\partial z \partial z}}
  \biggl\{-\Bigl(\frac{\partial^{2} f}{\partial z\partial z}\Bigr)^{2} + \frac{\partial f}{\partial z}\frac{\partial^{3} f}{\partial z \partial z \partial z}\biggr\}. 
\end{split}
\end{equation*}
Note that $J(0) = \lvert \frac{\partial f}{\partial z}(0)\rvert^{2} - t^{2}(a^{2} + b^{2}) \neq 0$ 
for $t \neq 0$ and $(a, b) \neq (0, 0)$. 
By the definitions of $G_t$ and $H_t$, we have 
\begin{equation*}
\begin{split}
&\{z \in U \setminus \{0\} \mid G_{t}(z) = H_{t}(z) = 0\} \\ 
= &\Bigl\{z \in U \mid J(z) = G_{1}(z) = G_{2}(z) = 
\frac{\partial (G_{1}, J)}{\partial (x, y)}(z) = \frac{\partial (G_{2}, J)}{\partial (x, y)}(z) =0 \Bigr\}. 
\end{split}
\end{equation*}
We show the existence of a linear deformation $f_t$ of $f$ which is an excellent map for generic $(a, b)$. 
\begin{lemma}
For a generic choice of $(a, b)$, 
$f_{t}|_{U}$ is an excellent map. 
\end{lemma}
\begin{proof}
We will show that there exists a deformation $f_t$ of $f$ 
such that $\{z \in U \setminus \{0\} \mid G_{t}(z) = H_{t}(z) = 0\}$ is empty. 
If $z$ satisfies $G_{t}(z) = H_{t}(z) = 0$, we have 
\begin{equation*}
\begin{split}
&\Bigl(\frac{\partial f}{\partial z}(z)\Bigr)^{2}\Bigl(\frac{\partial^{2} f}{\partial z \partial z}(z)\Bigr)^{2}\overline{\frac{\partial f}{\partial z}(z)}
  \biggl\{\overline{2\Bigl(\frac{\partial^{2} f}{\partial z\partial z}(z)\Bigr)^{2} - \frac{\partial f}{\partial z}(z)\frac{\partial^{3} f}{\partial z \partial z \partial z}(z)}\biggr\} \\
- &\Bigl(\frac{\partial f}{\partial z}(z)\Bigr)^{2}\overline{\frac{\partial f}{\partial z}(z)\Bigl(\frac{\partial^{2} f}{\partial z \partial z}(z)\Bigr)^{2}}
  \biggl\{-\Bigl(\frac{\partial^{2} f}{\partial z\partial z}(z)\Bigr)^{2} + \frac{\partial f}{\partial z}(z)\frac{\partial^{3} f}{\partial z \partial z \partial z}(z)\biggr\}  \\
=&\Bigl(\frac{\partial f}{\partial z}(z)\Bigr)^{2}\overline{\frac{\partial f}{\partial z}(z)}
\psi(z, \bar{z}) \\
=& 0, 
\end{split}
\end{equation*}
where $\psi(z, \bar{z}) = 
3\Bigl\lvert \frac{\partial^{2} f}{\partial z\partial z}\Bigr\rvert^{4} - 
\Bigl(\frac{\partial^{2} f}{\partial z \partial z}\Bigr)^{2}\overline{\frac{\partial f}{\partial z}\frac{\partial^{3} f}{\partial z \partial z \partial z}} -
\Bigl(\overline{\frac{\partial^{2} f}{\partial z \partial z}}\Bigr)^{2}\frac{\partial f}{\partial z}\frac{\partial^{3} f}{\partial z \partial z \partial z}$. 
Note that $\psi$ is a real-valued polynomial function. 
Since $\psi^{-1}(0)$ is a $1$-dimensional algebraic set, $\psi^{-1}(0)$ has finitely many branches which 
depend only on $f(z)$. On $U$, each branch of $\psi^{-1}(0)$ is given by a convergent power series 
\[
\xi_{m}(u) = \Bigl(\textstyle\sum_{\ell}c_{1,\ell}u^{\ell}, \textstyle\sum_{\ell}c_{2,\ell}u^{\ell}\Bigr), 
\]
where $0\leq u \ll 1$ for $m=1,\dots,d$. 
By Lemma $1$, the set $\{u \neq 0 \mid \phi_{1}(\xi_{m}(u)) = \phi_{2}(\xi_{m}(u)) =~0\}$ 
is empty for $m=1,\dots,d$. 
Since $\psi^{-1}(0) \cap U$ has finitely many branches, 
we can choose coefficients $a$ and $b$ of $\Phi$ such that 
\begin{equation*}
\Phi(\xi_{m}(u)) = a\phi_{1}(\xi_{m}(u)) + b\phi_{2}(\xi_{m}(u)) \neq 0 
\end{equation*}
for $0 < u \ll 1$ and $m=1,\dots,d$. 
Thus the intersection of $\{z\in U\setminus \{0\} \mid G_{1}(z)=G_{2}(z)=0\}$ and 
$\{z\in U\setminus \{0\} \mid \frac{\partial (G_{1}, J)}{\partial (x, y)}(z) = 
\frac{\partial (G_{2}, J)}{\partial (x, y)}(z) =0 \}$ 
is empty. 
By Proposition~$1$, 
the set of singularities of $f_{t}$ consists of either fold singularities or cusps. 
Therefore, 
$f_{t}|_{U}$ is an excellent map 
when $(a, b)$ satisfies $\Phi(\xi_{m}(u)) \neq 0$ for $0 < u \ll 1$ and $m=1,\dots,d$. 
\end{proof}
Let $w$ be a singularity of $f$ and $U_{w}$ be a sufficiently small neighborhood of $w$. 
By changing coordinates of $U_{w}$ and $f(U_{w})$, 
we may assume that $w = 0$ and $f(w)=0$. 
So we can apply Lemma $2$ to any singularity of $f$. 
Thus we can check that $f_t$ is an excellent map for $0 < \lvert t\rvert \ll 1$ if 
$a$ and $b$ are generic.

\section{Proof of Theorem $1$}
To calculate the number of cusps of $f_t$, we study zero points of $G_{t}$. 
\begin{lemma}
The set $\{z \in U \mid G_{t}(z)=0, z \neq 0\}$ is the set of 
positive simple roots of $G_{t}$ for  
$(a, b) \neq (0, 0)$ and $0 < \lvert t\rvert \ll 1$. 
\end{lemma}
\begin{proof}
If $z$ is a singularity of $G_{t}$, $z$ satisfies 
\[
\biggl\lvert -4i\frac{\partial f}{\partial z}\Bigl\lvert\frac{\partial^{2} f}{\partial z \partial z}\Bigr\rvert^{2} 
 + 2ti(a+ib)\frac{\partial^{3} f}{\partial z\partial z \partial z}\overline{\frac{\partial f}{\partial z}}\biggr\rvert = 
\biggl\lvert-2i\Bigl(\frac{\partial f}{\partial z}\Bigr)^{2}\overline{\frac{\partial^{3} f}{\partial z\partial z \partial z}} 
 + 2ti(a+ib)\Bigl\lvert \frac{\partial^{2} f}{\partial z \partial z}\Bigr\rvert^{2}\biggr\rvert,
\]
see \cite[Proposition 15]{O}. 
Assume that $z$ belongs to $G_{t}^{-1}(0)$. By the definition of $G_{t}$ and the above equation, we have 
\begin{equation}
\begin{split}
&\biggl\lvert 2\frac{\partial^{2} f}{\partial z \partial z}
\Bigl\lvert\frac{\partial f}{\partial z}\Bigr\rvert^{2}
\Bigl\lvert\frac{\partial^{2} f}{\partial z \partial z}\Bigr\rvert^{2} - 
\frac{\partial f}{\partial z}\frac{\partial^{3} f}{\partial z\partial z \partial z}
\Bigl\lvert\frac{\partial f}{\partial z}\Bigr\rvert^{2}
\overline{\frac{\partial^{2} f}{\partial z\partial z}}\biggr\rvert \\
= 
&\biggl\lvert \frac{\partial f}{\partial z}\frac{\partial^{2} f}{\partial z \partial z}
\Bigl\lvert\frac{\partial f}{\partial z}\Bigr\rvert^{2}\overline{\frac{\partial^{3} f}{\partial z\partial z \partial z}} - 
\Bigl(\frac{\partial f}{\partial z}\Bigr)^{2}\Bigl\lvert\frac{\partial^{2} f}{\partial z \partial z}\Bigr\rvert^{2}
\overline{\frac{\partial^{2} f}{\partial z\partial z}}\biggr\rvert. 
\end{split}
\end{equation}
Let $k$ be the multiplicity of $f$ at the origin. 
Then $f(z)$ has the following form: 
\[
f(z) = cz^{k} + (\text{higher terms}). 
\]
By equation $(3)$, we have 
\[
\lvert z^{5k-8}\rvert\lvert k^{6}(k-1)^{2}c\lvert c\rvert^{4} + (\text{higher terms})\rvert 
= \lvert z^{5k-8}\rvert\lvert -k^{5}(k-1)^{2}c\lvert c\rvert^{4} + (\text{higher terms})\rvert. 
\]
Since $U$ is sufficiently small and $k$ is greater than $1$, 
the above equation does not hold in $U\setminus \{0\}$. 
So we can show that 
\[
\Bigl\lvert \frac{\partial G_{t}}{\partial z}(z)\Bigr\rvert > 
\Bigl\lvert \frac{\partial G_{t}}{\partial \bar{z}}(z)\Bigr\rvert 
\]
for any $z \in (U\setminus\{0\}) \cap G_{t}^{-1}(0)$. 
Thus zero points of $G_{t}$ except for the origin are positive simple. 
\end{proof}

Assume that $f_t$ is an excellent map for $0 < \lvert t\rvert \ll 1$. 
We prove Theorem $1$.

\begin{proof}[Proof of Theorem 1]
By Proposition~$1$, 
the number of cusps of $f_{t}|_{U}$ is equal to 
the number of $\{z \in U \mid G_{t}(z) = 0, z\neq 0\}$. 
Set $\{z \in U \mid G_{t}(z) = 0, z\neq 0\} = \{w_{1}, \dots, w_{\nu} \}$. 
We denote the multiplicity of sign by $m_{s}(G_{t}, w_{j})$ for $j=1,\dots, \nu $. 
By \cite[Proposition 16]{O} and Lemma $3$, we have 
\[
\Bigl(\textstyle\sum_{j=1}^{\nu}m_{s}(G_{t}, w_{j})\Bigr) + m_{s}(G_{t}, 0) = 
\nu + m_{s}(G_{t}, 0) = m_{s}(G_{0}, 0). 
\]
The multiplicity $m_{s}(G_{0}, 0)$ is equal to  
\begin{equation*}
\begin{split}
&\deg 
\biggl(-2i\Bigl(\frac{\partial f}{\partial z}\Bigr)^{2}\overline{\frac{\partial^{2} f}{\partial z \partial z}}\biggr) 
\biggr/ \biggl\lvert-2i\Bigl(\frac{\partial f}{\partial z}\Bigr)^{2}\overline{\frac{\partial^{2} f}{\partial z \partial z}}\biggr\rvert 
: S^{1}_{\varepsilon}(0) \rightarrow S^1 \biggr) \\
&= 2(k-1)-(k-2) = k, 
\end{split}
\end{equation*}
where $S^{1}_{\varepsilon}(0) = \{z \in U \mid \lvert z\rvert = \varepsilon \}$ and 
$0 < \varepsilon \ll 1$. 
By the definition of $G_{t}$, for any $t \neq 0$, $m_{s}(G_{t}, 0)$ is equal to 
\begin{equation*}
\begin{split}
&\deg \biggl( 
2ti(a+ib)\frac{\partial^{2} f}{\partial z \partial z}\overline{\frac{\partial f}{\partial z}}
\biggr/ \biggl\lvert 
2ti(a+ib)\frac{\partial^{2} f}{\partial z \partial z}\overline{\frac{\partial f}{\partial z}}\biggr\rvert
: S^{1}_{\varepsilon_{t}}(0) \rightarrow S^{1} \biggr) \\
&= k-2-(k-1) = -1,
\end{split}
\end{equation*}
where $0 < \varepsilon_{t} \ll \varepsilon$. 
Thus the number of cusps of $f_{t}|_{U}$ is equal to $k+1$. 
\end{proof}

\begin{proof}[Proof of Corollary 1]
Set $\frac{\partial f}{\partial z} = n\prod_{j=1}^{\ell}(z-w_{j})^{m_{j}}$. 
Let $U_{j}$ be a sufficiently small neighborhood of $w_j$. 
By the same argument as in the proof of Theorem $1$,
the number of cusps of $f_{t}|_{U_{j}}$ is equal to $m_{j} + 2$. 
Note that $\textstyle \sum_{j=1}^{\ell}m_{j} = n-1$. 
Then the number of cusps of $f_t$ is equal to 
\begin{equation*}
\begin{split}
\textstyle \sum_{j=1}^{\ell}(m_{j} + 2) &= (\textstyle \sum_{j=1}^{\ell}m_{j}) + 2\ell \\
                             &= n-1 + 2\ell.
\end{split}
\end{equation*}
Since the number $\ell$ of singularities of $f$ belongs to $[1, n-1]$, 
the number of cusps of $f_t$ belongs to $[n+1, 3n-3]$. 
\end{proof}

Note that $n \geq 2$. 
We give an application of Theorem $1$. 
\begin{corollary}
Let $f_t$ be a linear deformation of a complex polynomial $f$. 
Assume that $f_t$ is an excellent map for $0 < \lvert t\rvert \ll 1$. 
Then the number of cusps of $f_t$ is at least three. 
\end{corollary}
\begin{proof}
By the change of coordinates, we may assume that the origin $0$ is a singularity of $f$ 
and $f(0) = 0$. 
Then the multiplicity of $f$ at $0$ is greater than $1$. 
By Theorem $1$, the number of cusps of $f_{t}|_{U}$ is at least three, where $U$ is a sufficiently small 
neighborhood of $0$. 
\end{proof}


\subsection{Deformations of $f_t$}
Let $f_t$ be a linear deformation of $f$ which is an excellent map. 
We fix $a, b$ and $t$. 
Let $g(z, \bar{z})$ be a mixed polynomial which satisfies 
$\frac{\partial g}{\partial z}(0) = \frac{\partial g}{\partial \bar{z}}(0) = 0$. 
In this subsection, we study a deformation of $f_t$: 
\[
f_{t,s}(z) := f(z) + t(a+ib)\bar{z} + sg(z, \bar{z}), 
\]
where $0 < \lvert s\rvert \ll \lvert t\rvert \ll 1$. 
\begin{theorem}
The set of singularities of $f_{t, s}$ consists of either fold singularities or cusps 
and 
the number of cusps of $f_{t, s}$ is constant for $0 \leq \lvert s\rvert \ll \lvert t\rvert \ll 1$. 
\end{theorem}
\begin{proof}
Let $w$ be a singularity of $f$ and $U_{w}$ be a sufficiently small neighborhood of $w$. 
Set 
$J_{s} = \frac{\partial (\Re f_{t,s}, \Im f_{t,s})}{\partial (x, y)}, 
G_{1 ,s} = \frac{\partial (\Re f_{t,s}, J_{s})}{\partial (x, y)}$ 
and $G_{2 ,s} = \frac{\partial (\Im f_{t,s}, J_{s})}{\partial (x, y)}$. 
We define mixed polynomials $G_{t,s}$ and $H_{t,s}$ as follows: 
\[
G_{t,s} := G_{1 ,s} + iG_{2 ,s}, \ \ \ 
H_{t,s} := 
\frac{\partial (G_{1,s}, J_{s})}{\partial (x, y)} + i\frac{\partial (G_{2,s}, J_{s})}{\partial (x, y)}. 
\]
Since the set of singularities of $f_{t,0}$ consists of either fold singularities or cusps, we have 
\[
\{z \in U_{w} \mid J_{0}(z) = G_{t,0} = H_{t,0} = 0\} = \emptyset. 
\]
Then there exists a positive real number $s_{0}$ such that 
\[
\{z \in U_{w} \mid J_{s}(z) = G_{t,s} = H_{t,s} = 0\} = \emptyset, 
\]
for any $0 \leq \lvert s\rvert \leq s_{0}$. Thus any singularity of $f_{t, s}$ is a fold singularity or a cusp. 
By the definition of $J_{t,s}$, the origin $0$ is a regular point of $f_{t,s}$. 
Since the set $\{z \in \Bbb{C} \mid G_{t,0}(z) = 0, z \neq 0\}$ is the set of positive simple roots of $f_{t,0}$, 
$\{z \in \Bbb{C} \mid G_{t,s}(z) = 0, z \neq 0\}$ is also the set of positive simple roots of $f_{t,s}$ for 
$0 \leq \lvert s\rvert \leq s_{0}$. 
By \cite[Proposition 16]{O}, the number of cusps of $f_{t, s}$ is constant for $0 \leq \lvert s\rvert \leq s_{0}$. 
\end{proof}

\subsection{Lower bounds of the numbers of cusps of non-linear deformations}
Let $h(z, \bar{z})$ be a mixed polynomial 
which satisfies $h(0) = 0$ and 
$\lvert \frac{\partial h}{\partial z}(0)\rvert \neq \lvert \frac{\partial h}{\partial \bar{z}}(0)\rvert$. 
We define a deformation $f_{t,h}$ of a complex polynomial $f$ as follows: 
\[
f_{t,h}(z) := f(z) + th(z, \bar{z}), 
\]
where $0 < \lvert t\rvert \ll 1$. Set $h_{1} = \Re h, h_{2} = \Im h$ and 
\[
J_{t, h} = \det \begin{pmatrix}
         \frac{\partial f_{1}}{\partial x} + t\frac{\partial h_{1}}{\partial x} & 
         \frac{\partial f_{1}}{\partial y} + t\frac{\partial h_{1}}{\partial y} \\
         -\frac{\partial f_{1}}{\partial y} + t\frac{\partial h_{2}}{\partial x} & 
         \frac{\partial f_{1}}{\partial x} + t\frac{\partial h_{2}}{\partial y}
         \end{pmatrix}. 
\]
Then any singularity of $f_{t,h}$ belongs to $J_{t, h}^{-1}(0)$. 
Assume that $f_{t,h}$ satisfies the following conditions: 
\renewcommand{\theenumi}{\roman{enumi}}
\begin{enumerate}
\item
$f_{t,h}$ 
is an excellent map for $0 < \lvert t\rvert \ll 1$,  \\
\item 
any cusp of $f_{t,h}$ is a simple root of $G_{t,h}$, where 
\begin{equation*}
\begin{split}
G_{t, h} &:=  \det \begin{pmatrix}
         \frac{\partial f_{1}}{\partial x} + t\frac{\partial h_{1}}{\partial x} & 
         \frac{\partial f_{1}}{\partial y} + t\frac{\partial h_{1}}{\partial y} \\   
         \frac{\partial J_{t, h}}{\partial x}  & \frac{\partial J_{t, h}}{\partial y}
         \end{pmatrix}
         + i\det \begin{pmatrix}
         -\frac{\partial f_{1}}{\partial y} + t\frac{\partial h_{2}}{\partial x} & 
         \frac{\partial f_{1}}{\partial x} + t\frac{\partial h_{2}}{\partial y} \\
         \frac{\partial J_{t, h}}{\partial x}  & \frac{\partial J_{t, h}}{\partial y} 
         \end{pmatrix}\\
    &= -2i\Bigl(\frac{\partial f}{\partial z} + t\frac{\partial h}{\partial z}\Bigr)\overline{\frac{\partial J_{t, h}}{\partial z}} 
    + 
  2ti\frac{\partial h}{\partial \bar{z}}\frac{\partial J_{t, h}}{\partial z}. 
\end{split}
\end{equation*}
\end{enumerate}
Since $f_{t,h}$ is an excellent map, 
the intersection of $J_{t, h}^{-1}(0)$ and $(\frac{\partial J_{t, h}}{\partial z})^{-1}(0)$ is empty 
by Proposition~$1$. 
Let $U$ be a sufficiently small neighborhood of the origin. 
Then the number of cusps of $f_{t,h}|_{U}$ is equal to 
the number of $\{z \in U \mid G_{t,h}(z) = 0, \frac{\partial J_{t, h}}{\partial z}(z) \neq 0 \}$. 
We define 
\begin{equation*}
\delta = \begin{cases}
         1 & \lvert \frac{\partial h}{\partial z}(0)\rvert > \lvert \frac{\partial h}{\partial \bar{z}}(0)\rvert \\
         -1 & \lvert \frac{\partial h}{\partial z}(0)\rvert < \lvert \frac{\partial h}{\partial \bar{z}}(0)\rvert
         \end{cases}.
\end{equation*}

\begin{theorem}
Let $f_{t,h}$ a deformation of a complex polynomial $f$ which satisfies the condition (i) and the condition (ii). 
Then the number of cusps of $f_{t,h}|_{U}$ is greater than or equal to $k - \delta$, where 
$k$ is the multiplicity of $f$ at the origin. 
\end{theorem}
\begin{proof}
Note that $m_{s}(G_{0,h}, 0) = 
m_{s}((\frac{\partial f}{\partial z})^{2}\overline{\frac{\partial^{2} f}{\partial z\partial z}}, 0) = k$. 
By \cite[Proposition 16]{O}, we have 
\begin{equation*}
\begin{split}
k = m_{s}(G_{0,h}, 0) &= \textstyle\sum_{\alpha \in G_{t,h}^{-1}(0)}m_{s}(G_{t,h}, \alpha) \\
&= 
\textstyle\sum_{\beta \in G_{t,h}^{-1}(0), \frac{\partial J_{t, h}}{\partial z}(\beta) \neq 0}m_{s}(G_{t,h}, \beta) 
+ \textstyle\sum_{\gamma \in (\frac{\partial J_{t, h}}{\partial z})^{-1}(0)} m_{s}(G_{t,h}, \gamma). 
\end{split}
\end{equation*}
Set $\deg \tilde{G} = \textstyle\sum_{\beta \in G_{t,h}^{-1}(0), \frac{\partial J_{t, h}}{\partial z}(\beta) \neq 0}m_{s}(G_{t,h}, \beta)$ and 
$\deg \tilde{J} = \textstyle\sum_{\gamma \in (\frac{\partial J_{t, h}}{\partial z})^{-1}(0)} m_{s}(G_{t,h}, \gamma)$. 
By the condition (ii), 
the number of cusps of $f_{t,h}|_{U}$ is greater than or equal to $\deg \tilde{G}$. 
By the definition of $J_{t, h}$, we have 
\[
\textstyle\sum_{\gamma \in (\frac{\partial J_{t, h}}{\partial z})^{-1}(0)} m_{s}(\frac{\partial J_{t, h}}{\partial z}, \gamma) = 
m_{s}(\frac{\partial J_{0, h}}{\partial z}, 0) = 
m_{s}(\frac{\partial^{2} f}{\partial z\partial z}\overline{\frac{\partial f}{\partial z}}, 0) =-1. 
\]
Since $\frac{\partial f}{\partial z}(0) = 0$ and 
$\lvert \frac{\partial h}{\partial z}(0)\rvert \neq \lvert \frac{\partial h}{\partial \bar{z}}(0)\rvert$, 
$\deg \tilde{J}$ is equal to $\delta$. 
Thus $\deg \tilde{G}$ 
is equal to $k - \delta$. 
\end{proof}

\section{Examples}
In this section, 
we construct a deformation of a complex polynomial which has 
$(n+1)$-cusps and also a deformation which has $(3n-3)$-cusps. 
\begin{Example}
Let $f(z) = z^{n}$ and $f_{t}(z) = z^{n} + t(a+ib)\bar{z}$ be a deformation of $f$ 
which satisfies the condition $(3)$. 
Then $G_{t}(z)$ is equal to 
\begin{equation*}
\begin{split}
G_{t}(z) &= -2in^{3}(n-1)z^{2n-2}\bar{z}^{n-2} + 2tn^{2}(n-1)(a+ib)z^{n-2}\bar{z}^{n-1} \\
      &= -2in^{2}(n-1)\lvert z\rvert^{2n-4}\{nz^{n} - t(a+ib)\bar{z}\}.
\end{split}
\end{equation*}
Set $z=re^{i\theta}$ and $a+ib = \tau e^{i\iota}$, where $\tau > 0$. 
Then we have 
\[
-2in^{2}(n-1)r^{2n-4}\{nr^{n}e^{ni\theta} - t\tau re^{i(\iota-\theta)}\}. 
\]
Assume that $z \neq 0$ and $G_{t}(z) = 0$. Then $z$ satisfies 
\[
r = \Bigl(\frac{t\tau}{n}\Bigr)^{\frac{1}{n-1}}, \ \ \ 
\theta = \frac{\iota + 2j\pi}{n+1}, 
\]
for $j = 0, \dots, n$. Thus the number of cusps of $f_t$ is equal to $n+1$. 
\end{Example}

\begin{Example}
Let $f(z) = z^{n} + z$. 
Then the number of singularities of $f$ is equal to $n-1$ and 
the multiplicity at each singularity of $f$ is equal to $2$. 
Let $f_{t}(z) = z^{n} + z + t(a+ib)\bar{z}$ be a deformation of $f$ which is an excellent map. 
By the same argument as in the proof of Corollary $1$, the number of cusps of $f_t$ is equal to $3n-3$. 
\end{Example}

\end{document}